\documentclass[12pt, reqno]{amsart}
\usepackage{amsmath}
\usepackage{amssymb}
\usepackage{amsthm}
\usepackage{xcolor}
\usepackage{mathrsfs}
\usepackage[abbrev]{amsrefs}
\usepackage[utf8]{inputenc}
\usepackage{ifthen}
\usepackage{enumitem}
\usepackage{bm}
\usepackage{graphicx}
\AtBeginDocument{\def\MR#1{}}
\newtheorem{thm}{}[section]
\newtheorem{theorem}[thm]{Theorem}
\newtheorem{corollary}[thm]{Corollary}
\newtheorem{lemma}[thm]{Lemma}
\newtheorem{proposition}[thm]{Proposition}

\theoremstyle{definition}

\theoremstyle{remark}

\usepackage{bbm}
\numberwithin{equation}{section}
\allowdisplaybreaks
\newcommand{\abs}[1]{\left\lvert#1\right\rvert}
\newcommand{\norm}[1]{\left\lVert#1\right\rVert}

\newcommand{\enbrace}[1]{\left\lbrace#1\right\rbrace}
\newcommand{\GG}{\ensuremath{\mathcal{G}}}
\newcommand{\Ft}{\ensuremath{\mathcal{F}}}
\newcommand{\Sym}{\ensuremath{\mathbb{S}}}
\newcommand{\HH}{\ensuremath{\mathbb{H}}}
\newcommand{\BI}{\ensuremath{\bm{B}}}
\newcommand{\AI}{\ensuremath{\bm{A}}}
\newcommand{\Ba}{\ensuremath{\mathbb{B}}}
\newcommand{\HB}{\ensuremath{\mathcal{H}}}
\newcommand{\Dy}{\ensuremath{\mathcal{D}}}
\newcommand{\Ind}{\ensuremath{\mathbbm{1}}}
\newcommand{\udf}{\ensuremath{\bm{\varphi}}}
\newcommand{\unc}{\ensuremath{\bm{k}}}
\newcommand{\aunc}{\ensuremath{\bm{\widetilde{k}}}}
\newcommand{\Jt}{\ensuremath{\mathcal{J}}}
\newcommand{\XB}{\ensuremath{\mathcal{X}}}
\newcommand{\YB}{\ensuremath{\mathcal{Y}}}
\newcommand{\UB}{\ensuremath{\mathcal{U}}}
\newcommand{\Nt}{\ensuremath{\mathcal{N}}}

\newcommand{\Ts}{\ensuremath{\mathcal{T}}}
\newcommand{\ts}{\ensuremath{\bm{t}}}
\newcommand{\vv}{\ensuremath{\bm{v}}}
\newcommand{\xx}{\ensuremath{\bm{x}}}
\newcommand{\yy}{\ensuremath{\bm{y}}}
\newcommand{\ee}{\ensuremath{\bm{e}}}
\newcommand{\zz}{\ensuremath{\bm{z}}}
\newcommand{\XX}{\ensuremath{\mathbb{X}}}
\newcommand{\YY}{\ensuremath{\mathbb{Y}}}
\newcommand{\BB}{\ensuremath{\mathcal{B}}}

\newcommand{\Id}{\ensuremath{\mathrm{Id}}}
\newcommand{\NN}{\ensuremath{\mathbb{N}}}
\newcommand{\ZZ}{\ensuremath{\mathbb{Z}}}
\newcommand{\FF}{\ensuremath{\mathbb{F}}}
\newcommand{\UU}{\ensuremath{\mathbb{U}}}
\newcommand{\LL}{\ensuremath{\mathcal{L}}}
\newcommand{\EB}{\ensuremath{\mathcal{E}}}
\newcommand{\supp}{\operatorname{supp}}

\newcommand\subsetsim{\mathrel{%
\ooalign{\raise0.2ex\hbox{$\subset$}\cr\hidewidth\raise-0.8ex\hbox{\scalebox{0.9}{$\sim$}}\hidewidth\cr}}}
\begin{document}
%------------------------------------------------------------------------
\title[Greedy-type bases in Tsirelson's space]{The structure of greedy-type bases in Tsirelson's space and its convexifications}
%------------------------------------------------------------------------
\author[F. Albiac]{Fernando Albiac}
\address{Department of Mathematics, Statistics and Computer Sciences, and InaMat$^2$\\ Universidad P\'ublica de Navarra\\
Pamplona 31006\\ Spain}
\email{fernando.albiac@unavarra.es}

\author[J. L. Ansorena]{Jos\'e L. Ansorena}
\address{Department of Mathematics and Computer Sciences\\
Universidad de La Rioja\\
Logro\~no 26004\\ Spain}
\email{joseluis.ansorena@unirioja.es}
%------------------------------------------------------------------------
%------------------------------------------------------------------------
%------------------------------------------------------------------------
\begin{abstract}
Tsirelson's space $\Ts$ made its appearance in Banach space theory in 1974 soon to become one of the most significant counterexamples in the theory. Its structure broke the ideal pattern that analysts had conceived for a generic Banach space, thus giving rise to the era of pathological examples. Since then, many authors have contributed to the  study of different aspects of this special space with an eye on better understanding its idiosyncrasies. In this paper we are concerned with the greedy-type basis structure of $\Ts$, a subject that had not been previously explored in the literature. More specifically, we show that Tsirelson's space and its convexifications $\Ts^{(p)}$ for $0<p<\infty$ have uncountably many non-equivalent greedy bases. We also investigate the conditional basis structure of spaces $\Ts^{(p)}$ in the range of $0<p<\infty$ and prove that they have uncountably many non-equivalent conditional almost greedy bases.
\end{abstract}
%------------------------------------------------------------------------
\thanks{The first-named author acknowledges the support of the Spanish Ministry for Science and Innovation under Grant PID2019-107701GB-I00 for \emph{Operators, lattices, and structure of Banach spaces}. The research of both authors was supported by the Spanish Ministry for Science, Innovation, and Universities under Grant PGC2018-095366-B-I00 for \emph{An\'alisis Vectorial, Multilineal y Aproximaci\'on}.}
%------------------------------------------------------------------------
\maketitle
%------------------------------------------------------------------------
\section{Introduction}\noindent
%------------------------------------------------------------------------
Since the appearance of Banach's dissertation in 1920 and the subsequent dissemination of his ideas twelve years later through the publication of \emph{Th\'eorie des op\'erations lin\'eares} \cite{Banach1932}, Banach space theory evolved during the 1960s to produce a likely picture of the structure of Banach spaces in which the spaces $\ell_{p}$ for $1\le p<\infty$ and $c_{0}$ were considered potential building blocks of any space. A question then arose as to whether every Banach space must contain a copy of one of these spaces. It was quite surprising when in 1974, Tsirelson \cite{Tsirelson1974} produced the first example of a Banach space not containing some $\ell_{p}$ ($1\le p<\infty$) or $c_{0}$.

The space nowadays known as \emph{Tsirelson's space}, herein denoted by $\Ts$, was introduced by Figiel and Johnson in 1974 \cite{FigielJohnson1974} to prove the existence of a superreflexive Banach space with an unconditional basis which does not contain a copy of $\ell_p$ for any value of $p\in[1,\infty)$ or $c_{0}$. As a matter of fact, the dual space $\Ts^*$ of $\Ts$ is the original space constructed by Tsirelson.

Despite its apparently strange construction, which uses an iteration method to define the norm, Tsirelson's space has turned out to be a remarkable springboard for further research. It has been studied in many different settings because it is an important source of examples and counterexamples in classical Banach space theory. For instance, the $2$-convexified $\Ts^{(2)}$ of Tsirelson's space, also introduced in \cite{FigielJohnson1974}, played a fundamental role in the study of the unconditional structure of Banach spaces when Bourgain et al.\ discovered in 1984 that it had a unique unconditional basis up to equivalence and permutation \cite{BCLT1985}, shattering thus all hopes of attaining a satisfactory classification of the Banach spaces with that property.

In this note we will investigate Tsirelson's space and its convexified analogues in the context of approximation theory. Although this is the first time that Tsirelson's space is studied in this context, where the problems that we consider have an extensive background, the questions we address are also of interest to the Banach space specialist since they provide an accurate description of the rich structure of conditional almost greedy bases in $\Ts$. By using a combination of abstract methods from Banach space theory together with a refinement of techniques from greedy bases we will provide further evidence of the fact that, in spite that Tsirelson's space does not contain copies of $\ell_{1}$, in a certain sense these two spaces 
can be very close to each other. More precisely, the structure of their conditional almost greedy bases follow the same pattern since, as it happens with $\ell_{1}$ (see \cite{AAW2021b}), the space $\Ts$ has uncountably many nonequivalent conditional almost greedy bases whose fundamental function is of the same order as the fundamental function of the canonical basis of $\ell_{1}$.

Throughout this paper we employ the standard notation and terminology commonly used in Banach space theory and approximation theory, as can be found, e.g., in the monographs \cites{AlbiacKalton2016,LinTza1979} or the recent article \cite{AABW2021}. Since by now the definition of Tsirelson's space is classical, we refer to the book of Casazza and Shura \cite{CasShu1989} from 1989 for its construction and elementary properties.

For the convenience of the reader and to fix the notation, we next gather together the most heavily used terminology.

A sequence space on a countable set $\Jt$ will be a quasi-Banach space for which the unit vector system $\EB_\Jt=(\ee_j)_{j\in\Jt}$ is a $1$-unconditional basis. If $\LL$ is a sequence space indexed on $\Jt$, then its square $\LL^2=\LL\oplus\LL$ is a sequence space naturally indexed on $\enbrace{1,2}\times \Jt$. We will say that a sequence space $\LL_1$ indexed on a set $\Jt_1$ is lattice isomorphic to a sequence space $\LL_2$ indexed on set $\Jt_2$ via a bijection $\pi\colon\Jt_2\to\Jt_1$ if the map $(a_n)_{n\in\Jt_1} \mapsto (a_{\pi(n)})_{n\in\Jt_2}$ restricts to a lattice isomorphism from $\LL_1$ onto $\LL_2$. Note that $\LL_1$ are $\LL_2$ are lattice isomorphic if and only if their unit vector systems $\EB_{\Jt_1}$ and $\EB_{\Jt_2}$ are permutatively equivalent.

The \emph{restriction} of a sequence space $\LL$ to a set $\Jt_0\subseteq\Jt$ is the sequence space consisting of all $f\in\FF^{\Jt_0}$ such that $\tilde{f}\in\LL$, where $\tilde{f}\in\FF^\Jt$ is the obvious extension of $f$.

In the most general setting, by a \emph{basis} of a Banach (or quasi-Banach) space $\XX$ we mean a norm-bounded countable family $\XB=(\xx_n)_{n\in\Nt}$ that generates the entire space $\XX$, in the sense that
\[
[\xx_n \colon n\in\Nt]=\XX,
\]
and for which there is a (unique) norm-bounded family $\XB^*=(\xx_n^*)_{n\in\Nt}$ in the dual space $\XX^*$ such that $(\xx_{n}, \xx_{n}^*)_{n\in\Nt}$ is a biorthogonal system. The sequence $\XB^*$ will be called the \emph{dual basis} of $\XB$.

For $f\in \XX$ we define the \emph{greedy ordering} for $f$ as the map $\rho\colon \Nt\to\Nt$ such that 
$\enbrace{n\colon \xx_{n}^*(f)\not=0}\subseteq \rho(\Nt)$ and such that if $j<k$ then
\[
\abs{\xx_{\rho(j)}^*(f)}>\abs{\xx_{\rho(k)}^*(f)}\quad \mbox{or}\quad \abs{\xx_{\rho(j)}^*(f)}=\abs{\xx_{\rho(k)}^*(f)}\; \mbox{and}\;\rho(j)<\rho(k).
\]
The \emph{$m$th greedy approximation} of $f$ is given by
\[
\GG_{m}(f)=\sum_{j=1}^{m}\xx_{\rho(j)}^*(f)\xx_{\rho(j)}.
\]
Konyagin and Temlyakov defined the basis $(\xx_n)_{n\in\Nt}$ to be \emph{greedy} if $\GG_{m}(f)$ is essentially the best $m$-term approximation to $f$ using the basis vectors, i.e., if there exists a constant $C\ge 1$ such that for all $f\in \XX$ and all $m\in \Nt$, we have
\begin{equation}\label{defgreedy}
\norm{ f-\GG_{m}(f)}\le C\inf\enbrace{\norm{ f-\sum_{n\in A}\alpha_n \xx_n}\colon \abs{A}=m, \alpha_n\in \FF }.
\end{equation}
Then they showed that greedy bases can be simply characterized as unconditional bases with the additional property of being \emph{democratic}, i.e., for some $\Delta\ge 1$ we have 
\[
\norm{\Ind_A} \le \Delta \norm{\Ind_A}\quad\mbox{whenever } \abs{A}\le \abs{B}.
\]
Here, and throughout the paper, given a basis $\XB=(\xx_n)_{n\in\Nt}$ of $\XX$ and $A\subseteq \Nt$ finite, we put
\[
\Ind_A=\Ind_A[\XB,\XX]=\sum_{n\in A} \xx_n.
\]
Then a basis $\XB$ is democratic if and only if there is a non-decreasing sequence $(s_m)_{m=1}^\infty$ in $(0,\infty)$ such that $\norm{ \Ind_A} \approx s_{\abs{A}}$ for $A\subseteq\Nt$ finite. Such a sequence $(s_m)_{m=1}^\infty$ must be equivalent to the \emph{fundamental function} of $\XB$, given by
\[
\udf[\XB,\XX](m)= \sup_{\abs{A}\le m} \norm{ \Ind_A[\XB,\XX] }, \quad m\in\NN.
\]
A democratic basis $\XB=(\xx_n)_{n\in\Nt}$ (or its fundamental function $\udf$) is said to have the \emph{lower regularity property} (LRP for short) if there are constants $C>0$ and $0<\alpha<1$ such that 
\[
C\udf(mn)\ge m^{\alpha}\udf(n), \quad m,n\in \Nt.
\] In turn, $\XB$ has the \emph{upper regularity property} (URP for short) if there are constants $C>0$ and $0<\beta<1$  such that
\[
\udf(n)\le C\left(\frac{n}{m}\right)^{\beta}\udf(m), \quad m\le n.
\]

In their groundbreaking article \cite{KoTe1999}, Konyagin and Temlyakov defined a basis $(\xx_n)_{n\in\Nt}$ to be \emph{quasi-greedy} if there exists a constant $C\ge 1$ such that $\norm{ \GG_{m}(f)} \le C\norm{f}$ for all $x\in \XX$ and $m\in \Nt$. Subsequently, Wojtaszczyk \cite{Woj2000} proved that these are precisely the bases for which $\lim_{m\to\infty}\GG_{m}(f)=f$ for all $f\in \XX$.

We also recall that a basis $(\xx_n)_{n\in \Nt}$ of a Banach space $\XX$ is \emph{almost greedy}  if there exists a constant $C\ge 1$ such that for all $f\in \XX$ and all $m\in \Nt$ we have
\begin{equation}\label{defalmostgreedy}
\norm{ f-\GG_{m}(f)}\le C\inf\enbrace{\norm{f-\sum_{n\in A}\xx_n^*(f)\xx_n}\colon \abs{A}=m}.
\end{equation}
Comparison with \eqref{defgreedy} shows that this is formally a weaker condition than bein greedy: in \eqref{defgreedy} the infimum is taken over all possible $m$-term approximations, while in \eqref{defalmostgreedy} only projections of $f$ onto the basis vectors are considered. It was proved in \cite{DKKT2003} that $(\xx_n)_{n\in \Nt}$ is almost greedy if and only if $(\xx_n)_{n\in \Nt}$ is quasi-greedy and democratic.

The symbol $\XX\unlhd \YY$ means that the space $\XX$ is isomorphic to a complemented subspace of $\YY$. Quantitatively, we write $\XX\unlhd_C \YY$ if there are linear operators $S\colon\XX\to \YY$ and $T\colon \YY\to\XX$ such that $T\circ S=\Id_\XX$ and $\norm{ S} \, \norm{ T} \le C$. If we also have $S\circ T=\Id_\YY$, the spaces $\XX$ and $\YY$ are isomorphic, and we write $\XX\simeq_C\YY$, or $\XX\simeq\YY$ if the constant $C$ is irrelevant.

Other more specific terminology will be introduced in context.

%------------------------------------------------------------------------
\section{Preparatory results}\noindent
%------------------------------------------------------------------------
Tsirelson's space $\Ts$ is a Banach space for which the the canonical unit vectors form a normalized $1$-unconditional basis of $\Ts$ that we will denote by $(\ts_n)_{n=1}^\infty$. The lattice structure induced on $\Ts$ by its canonical basis will be crucial in our arguments. Let us get started by recalling a well-known fact for further reference.

\begin{theorem}[see \cite{CasShu1989}*{Proposition I.12}]\label{thm:CSSQ}
Tsirelson's space $\Ts$ is lattice isomorphic to its square via the bijection $\pi\colon \enbrace{1,2} \times \NN\to\NN$ given by $\pi(j,n)=2n+j-2$.
\end{theorem}

We also need to have a good understanding of the behavior of certain subsequences of $(\ts_n)_{n=1}^\infty$, such as the one illustrated in the following theorem.

\begin{theorem}[see \cite{CasShu1989}*{Corollary II.5}]\label{thm:CSB}
There is a constant $C$ such that for every increasing map $\varphi\colon\NN\to\NN$ and every sequence $(f_k)_{k=1}^\infty$ with $f_k\in[\ts_n \colon 1+\varphi(k-1) \le n \le \varphi(k)]$, we have
\[
\frac{1}{C} \norm{ \sum_{k=1}^\infty f_k} \le \norm{ \sum_{k=1}^{\infty} \norm{ f_k}\, \ts_{\varphi(k)} } \le C \norm{ \sum_{k=1}^\infty f_k}.
\]
\end{theorem}

To facilitate the description of subsequences $(\ts_{n_{k}})_{k=1}^{\infty}$ of $(\ts_{n})_{n=1}^{\infty}$ with $k_{n+1}-k_{n}$ extremely large, which are nevertheless equivalent to $(\ts_{n})_{n=1}^{\infty}$, we need to introduce a compact notation from logic. Following \cite{CasShu1989} we call the \emph{fast growing hierarchy} to the family of `slowly rapidly growing' functions introduced by Smory\'{n}ski \cite{Smorynski1985}. Specifically, we recursively define a sequence $(F_n)_{n=0}^\infty$ of $\NN$-valued functions on the natural numbers by
\begin{align*}
F_0(j)&=j+1, \quad j\in\NN;\\
F_{n}(j)&=F_{n-1}^{(j)}(j), \quad j,\, n\in\NN.
\end{align*}
We say that a function $\varphi\colon\NN\to\NN$ is \emph{dominated by the fast growing hierarchy} is there are $n\in\NN\cup\enbrace{0}$ and $j_0\in\NN$ such that $\varphi(j) \le F_n(j)$ for all $j \ge j_0$.

We next summarize the basic properties of this notion. In general, if $X$ is a set and $f\colon X \to X$ is a map, we can recursively define maps $f^{(n)}\colon X \to X$ by
\[
f^{(0)}=\Id_X, \quad f^{(n)}=f^{(n-1)} \circ f.
\]
If $\sim$ is a relation on $X$ and $f(x)\sim f(y)$ whenever $x\sim y$, then $f^{(n)}\sim f^{(n)}(y)$ for all $n\in\NN$ and all $x$ and $y\in X$ with $x\sim y$.

\begin{lemma}\label{lem:Zero}
Suppose $f\colon\NN\to\NN$ is an increasing map with $f(1)\ge k$ for some $k\in\NN$. Then $f^{(n)}$ is increasing for every $n\in\NN$. Moreover,
\[
f^{(n)}(j)\ge (k-1)(n-m)+ f^{(m)}(j), \quad j\in\NN, \; 0\le m \le n.
\]
\end{lemma}

\begin{proof}
It is a straightforward induction on $n$.
\end{proof}

\begin{lemma}\label{lem:One}
Let $(F_n)_{n=0}^\infty$ denote the fast growing hierarchy. Then:
\begin{enumerate}[label=(\roman*), leftmargin=*, widest=ii]
\item\label{FGR:A} $F_n$ is an increasing function with $F_n(1)\ge 2$ for all $n\in\NN$.
\item\label{FGR:B} For each $j\in\NN$, $(F_n(j))_{n=0}^\infty$ is a non-decreasing sequence.
\item\label{FGR:C} $F_1(j)=2j$ and $F_2(j)=j2^j$ for all $j\in\NN$.
\end{enumerate}
\end{lemma}

\begin{proof}
We prove \ref{FGR:A} by induction. The result  obviously holds for $n=0$. Let $n\in\NN$ and assume that the result holds for $n-1$. We have $F_n(1)=F_{n-1}(1)=2$. Given $j\in\NN$, applying Lemma~\ref{lem:Zero} with $k=2$ gives.
\[
F_n(j+1)
=F_{n-1}^{(j+1)}(j+1)
\ge 1+ F_{n-1}^{(j)}(j+1)
\ge 2+F_{n-1}^{(j)}(j)=2+F_{n}(j).
\]

To prove \ref{FGR:B} we pick natural numbers $j$ and $n$. By \ref{FGR:A} and Lemma~\ref{lem:Zero},
\[
F_{n}(j)=F_{n-1}^{(j)}(j)\ge (j-1) +F_{n-1}(j).
\]

Let us prove \ref{FGR:C}. By induction, $F_0^{(k)}(j)=j+k$ for all $j\in\NN$ and $k\in\NN\cup\enbrace{0}$. Consequently, $F_1(j)=2j$ for all $j\in\NN$. Then, applying induction to $F_1$ gives $F_1^{(k)}(j)=j 2^k$ for all $j\in\NN$ and $k\in\NN\cup\enbrace{0}$.
\end{proof}

\begin{lemma}\label{lem:FGHProp}
Suppose that $\varphi$ and $\psi$ are dominated by the fast growing hierarchy. Then the following functions are dominated by the fast growing hierarchy:
\begin{enumerate}[label=(\roman*)]
\item\label{FGH1} $\psi\circ\varphi$;
\item\label{FGH2} $\varphi+\psi$;
\item\label{FGH3} $\varphi\psi$;
\item\label{FGH4} $\varphi^\psi$, in particular $R^\psi$ for $R\in\NN$; and
\item\label{FGH5} the function $\rho$ defined by $\rho(j)=\sum_{i=1}^j \varphi(i)$.
\end{enumerate}
\end{lemma}

\begin{proof}
Suppose that $\varphi(j)\le F_n(j)$ for all $j\ge j_0$ and that $\psi(j)\le F_m(j)$ for all $j\ge j_1$. Set $k=1+\max\enbrace{n,m}$. Let $j_2$ be such that $F_k(j_2)\ge \psi(j)$ for all $j\in\NN[j_1]$. Pick $j\ge \max\enbrace{2,j_0,j_2}$. If $\varphi(j) \ge j_1$, by Parts~\ref{FGR:A} and \ref{FGR:B} of Lemma~\ref{lem:One},
\[
\psi(\varphi(j))\le F_m(\varphi(j))\le F_m(F_n(j))\le F_{k-1}^{(2)}(j)\le F_{k-1}^{(j)}(j)=F_k(j).
\]
In turn, if $\varphi(j)<j_1$, then $\psi(\varphi(j))\le F_k(j_2)\le F_k(j)$. This takes care of \ref{FGH1}, and the other parts can be derived from it, as we next show.

By Part~\ref{FGR:B} of Lemma~\ref{lem:One} $\eta=\max\enbrace{\varphi,\psi}$ is dominated by the fast growing hierarchy. By Part~\ref{FGR:C} of Lemma~\ref{lem:One} the function $\delta$ given by $\delta(j)=2j$ is dominated by the fast growing hierarchy. Since $\varphi+\psi\le \delta\circ\eta$, \ref{FGH2} holds. By Part~\ref{FGR:C} of Lemma~\ref{lem:One} the function $\epsilon$ given by $\epsilon(j)=2^j$ is dominated by the fast growing hierarchy. Since $\varphi\psi\le \epsilon\circ (f+g)$, \ref{FGH3} holds. Since $\varphi^\psi \le \eta\circ (\varphi\psi)$, \ref{FGH4} holds. Let $j_0\in\NN$ and $n\in\NN$ be such that $\varphi(j)\le F_n(j)$ for all $j\ge j_0$. Then, $\varphi\le \gamma\circ F_n$, where
\[
\gamma(j)=\sum_{j=1}^{j_0-1} \varphi(j) + \max\enbrace{j-j_0+1,0}F_n(j), \quad j\in\NN,
\]
Since $\gamma$ is dominated by the fast growing hierarchy, \ref{FGH5} holds.
\end{proof}

We conclude our primer on slowly rapidly increasing functions by showing the existence of a continuum of essentially different ones.

\begin{theorem}\label{thm:FGHContinuum}
There is a continuum $\Ft$ of $\NN$-valued increasing functions on $\NN$ such that:
\begin{itemize}[leftmargin=*]
\item $\varphi(j) \le 3^{j+1}$ for all $\varphi\in\Ft$ and $j\in\NN$; and
\item if $\varphi$ and $\psi$ are different elements of $\Ft$ then $\varphi(\NN)\cap\psi(\NN)$ is finite.
\end{itemize}
\end{theorem}

\begin{proof}
We will use expansions in base $3$ of numbers in $[0,1]$. Given $\varepsilon=(\varepsilon_n)_{n=1}^\infty=\enbrace{1,2}^\NN$, the series
\[
\sum_{n=1}^\infty \varepsilon_n 3^{-n}
\]
defines a number in $(0,1]$, and different sequences in $\varepsilon=\enbrace{1,2}^\NN$ give rise to different numbers in $(0,1]$. Therefore, the family of increasing functions from $\NN$ into $
X=(0,1]\cap\left( \cup_{j=0}^\infty 2^{-j}\ZZ\right)$
$(s_\varepsilon)_{\varepsilon=\enbrace{1,2}^\NN}$
defined by
\[
s_\varepsilon(j)=\sum_{n=1}^j\varepsilon_n 3^{-n}, \quad \varepsilon=(\varepsilon_n)_{n=1}^\infty,\; j\in\NN,
\]
has the wished-for intersection property. Given $x\in X$, set
\[
j(x)=\min\enbrace{j\colon 3^j x\in\ZZ}\quad \mbox{and} \quad \nu(x)= 3^{j(x)} (1+2 x).
\]
Since the function $\nu\colon X \to \NN$ is increasing, the maps
\[
\nu\circ s_\varepsilon, \quad \varepsilon=\enbrace{1,2}^\NN
\]
are a family of increasing functions from $\NN$ into $\NN$ with the desired intersection property. To finish the proof we just need to notice that $\nu ( s_\varepsilon(j) )\le 3^{j+1}$ for all $\varepsilon=\enbrace{1,2}^\NN$ and $j\in\NN$.
\end{proof}

The following result shows surprising similarities between the basis of Tsirelson's space and the canonical basis of $\ell_1$.

\begin{theorem}[see \cite{CasShu1989}*{Theorems IV.a.1 and IV.c.1}]\label{thm:CSA}
Suppose $\varphi\colon\NN\to\NN$ is increasing. Then the following are equivalent:
\begin{enumerate}[label=(\roman*),leftmargin=*,widest=iii]
\item The sequences $(\ts_n)_{n=1}^\infty$ and $(\ts_{\varphi(j)})_{j=1}^\infty$ are equivalent.
\item There is a constant $C$ such that
$(\ts_{n+\varphi(j)})_{n=1}^{\varphi(j+1)-\varphi(j)}$ is $C$-equivalent to the unit vector system of $\ell_1^{\,\varphi(j+1)-\varphi(j)}$ for all $j\in\NN$.
\item $\varphi$ is dominated by the fast growing hierarchy.
\end{enumerate}
\end{theorem}

%------------------------------------------------------------------------
\section{Direct sums in the sense of Tsirelson's space}\noindent
%------------------------------------------------------------------------
Given a sequence space $\LL$, and a family $(\XX_j, \norm{\cdot}_{\XX_{j}})_{j\in\Jt}$ of quasi-Banach spaces with moduli of concavity uniformly bounded, the space
\[
\left(\bigoplus_{j\in\Jt} \XX_j\right)_\LL=\enbrace{f=(f_j)_{j\in\Jt}\in\prod_{j\in\Jt} \XX_j\colon \norm{ (\norm{ f_j}_{\XX_{j}})_{n\in\Jt} }_\LL<\infty}
\]
is a quasi-Banach space with the quasi-norm
\[
\norm{ f}= \norm{ (\norm{ f_j}_{\XX_{j}})_{n\in\Jt}}.
\]
For each $k\in\Jt$ let $L_k\colon\XX_k \to (\bigoplus_{j\in\Jt} \XX_j)_\LL$ be the canonical embedding. If for each each $j\in\Jt$, $\XB_j=(\xx_{j,n})_{n\in\Nt_j}$ is a basis of $\XX_j$ and both $\XB_j$ and $\XB_j^*$ are uniformly bounded, then the sequence
\[
\left(\bigoplus_{j\in\Jt} \XB_j\right)_\LL
=\left( L_j(\xx_{j,n})\right)_{n\in\Nt_j,\; j\in\Jt}
\]
is a basis of the space $(\bigoplus_{j\in\Jt} \XX_j)_\LL$. If $\XB_j$ is normalized for all $j\in\Jt$, so is $(\bigoplus_{j\in\Jt} \XB_j)_\LL$. If $\XB_j$ is a $K$-unconditional basis for each $j\in\Jt$, so is $(\bigoplus_{j\in\Jt} \XB_j)_\LL$. In particular, if each $\XX_j$ is a sequence space of a set $\Nt_j$, then, $(\bigoplus_{j\in\Jt} \XB_j)_\LL$ is a sequence space on the countable set
\[
\Nt=\cup_{j\in\Jt} \enbrace{ j} \times \Nt_j,
\]
modulus the natural identification of the unit vector $\ee_{j,n}$, $(j,n)\in\Nt$, with the vector $L_j(\ee_n)$. If $\XB_j$ is $C$-quasi-greedy for all $j\in\Jt$, then $\XB$ is $C$-quasi-greedy.

Set $s=\abs{\Nt}$ and suppose that $\Nt_j=\NN[s_j]$ for some sequence $(s_j)_{j=1}^\infty$ in $\NN\cup\enbrace{\infty}$, where
\[
\NN[m]=\ZZ\cap[1,m],\quad m\in\NN.
\]
Then, if $\XB_j$ is a Schauder basis with constant $K$ for each $j\in\Jt$, and $\pi=(\pi_1,\pi_2)\colon \NN[s]\to\Nt$ is a bijection with $\pi_2(m) \le \pi_2(n)$ whenever $m\le n$ and $\pi_1(m)=\pi_1(n)$, then $(L_{\pi_1(n)}( \xx_{\pi(n)}))_{n=1}^s$ is a Schauder basis with constant $K$.

If $\XX_j=\XX$ for all $j\in\Jt$, we set $\LL(\XX)=(\bigoplus_{j\in\Jt} \XX_j)_\LL$. Similarly, if $\XB_j=\XB$ for all $j\in\Jt$, we set $\LL(\XB) =(\bigoplus_{j\in\Jt} \XX_j)_\LL$.

Let us next present a family of ``Tsirelson-like'' spaces, called convexified Tsirelson spaces, also introduced in the aforementioned paper \cite{FigielJohnson1974} by Figiel and Johnson. For broader generality we include the nonlocally convex members of the family in the following definition (cf.\ \cite{CasShu1989}*{p.\ 116 ff.}).

Fix $0<p<\infty$. The \emph{$p$-convexification of $\Ts$}, denoted $\Ts^{(p)}$, is the set of all sequences $f=(a_n)_{n=1}^\infty$ such that $\abs{f}^p:=(\abs{a_n}^p)_{n=1}^\infty\in\Ts$, equipped with the quasi-norm (norm if $p\ge 1$)
\[
\norm{ f }_{\Ts^{(p)}}=\norm{ \sum_{n=1}^{\infty}\abs{a_{n}}^{p}\, \ts_{n}}_{\Ts}^{1/p}.
\]

Of course, for $p=1$ we obtain the usual space $\Ts$.

In general, given $0<p<\infty$ the $p$-convexification of a sequence space $\LL$ is defined by
\[
\LL^{(p)}=\enbrace{ f\in \FF^\Jt \colon \abs{f}^p \in\LL}.
\]
It is routine to check that $\LL^{(p)}$, endowed with the quasi-norm $f\mapsto \norm{ \abs{f}^p }_\LL^{1/p}$, is a sequence space. Clearly, $\LL^{(1)}=\LL$, $\LL^{(pq)}=(\LL^{(p)})^{(q)}$ for all $p$, $q\in(0,\infty)$. If $(\XX_j)_{j\in\Jt}$ is a family of sequence spaces, then
\begin{equation}\label{eq:directsum+pconvex}
\left(\left(\bigoplus_{j\in\Jt} \XX_j\right)_\LL\right)^{(p)}
=\left( \bigoplus_{j\in\Jt} \XX_j^{(p)}\right)_{\LL^{(p)}}.
\end{equation}

Our constructions of special types of bases in the spaces $\Ts^{(p)}$ in Sections~\ref{sect:greedy} and \ref{Sect:AlmostGreedy} rely on Theorem~\ref{thm:TsirelsonIsomorphism}, which roughly speaking tells us that, despite the fact that the canonical bases of $\Ts^{(p)}$ and $\ell_{p}$ are not equivalent, they move away one from another very slowly.

\begin{theorem}\label{thm:TsirelsonIsomorphism}
Let $0<p<\infty$. Suppose that $\varphi\colon\NN\to\NN$ is a function dominated by the fast growing hierarchy. Then $( \bigoplus_{j=1}^\infty \ell_p^{\varphi(j)})_{\Ts^{(p)}}$ is lattice isomorphic to $\Ts^{(p)}$ via the canonical bijection
\[
\pi\colon\NN\to \enbrace{(j,n)\in\NN^2 \colon n\le \varphi(j)}
\]
given by $\pi(k)=(j,n)$ if $k=n+\sum_{i=1}^{j-1} \varphi(i)$.
\end{theorem}

\begin{proof}
By \eqref{eq:directsum+pconvex} it suffices to consider the case $p=1$. By Part~\ref{FGH5} of Lemma~\ref{lem:FGHProp}, the function $\psi\colon\NN\to\NN$ defined $\psi(j)=\sum_{i=1}^j \varphi(i)$ is dominated by the fast growing hierarchy. Let $\LL$ be the sequence space corresponding to the subbasis $(\ts_{\psi(j)})_{j=1}^\infty$ of the unit vector system of $\Ts$. For each $j\in\NN$, let $\XX_j$ the sequence space corresponding to the finite subbasis $(\ts_{n})_{n=\psi(j-1)+1}^{\psi(j)}$. Since the function $j\mapsto 1+\psi(j-1)$ also is dominated by the fast growing hierarchy, we infer from Theorem~\ref{thm:CSA} that
\[
\left( \bigoplus_{j=1}^\infty \ell_1^{\varphi(j)}\right)_{\Ts}=\left( \bigoplus_{j=1}^\infty \XX_j \right)_{\LL}
\]
up to an equivalent norm. In turn, an application of by Theorem~\ref{thm:CSB} gives that the mapping $(a_n)_{n=1}^\infty \mapsto (a_{\pi(n)})_{n=1}^\infty$ restricts to a lattice isomorphism from $\Ts$ onto $\left( \bigoplus_{j=1}^\infty \XX_j \right)_{\LL}$.
\end{proof}

We next see an application of Theorem~\ref{thm:TsirelsonIsomorphism} to direct sums involving the Haar system in $[0,1]$, which we will apply in Section~\ref{sect:greedy}.

Given $n\in\NN$, let $\Dy_n$ be the set consisting of all dyadic intervals of length is $2^{-n+1}$ contained in $[0,1]$. Note that $\abs{\Dy_n}=2^{n-1}$. Given $I\in\Dy:=\cup_{n=1}^\infty \Dy_n$, $h_I^{(p)}$ denotes the $L_p$-normalized Haar function whose support is $I$. Given $A\subseteq\NN$, we denote by $L_p^A$ the $(2^{n-1}\abs{A})$-dimensional subspace of $L_p=L_p([0,1])$ spanned by
\[
\HB^{(p)}_{A}=\left(h^{(p)}_I\right)_{I\in\Dy_A},
\]
where $\Dy_A=\cup_{n\in A}\Dy_n$. For each $n\in\NN$ let $\Sigma_n$ denote the finite $\sigma$-algebra generated by $\Dy_n$. Note that, in the case when $A$ is finite,
\begin{equation}\label{eq:DyadicEmbedding}
L_p^A \subseteq L_p(\Sigma_n), \quad n=\max(A).
\end{equation}

Given a sequence $\AI=(A_j)_{j=1}^\infty$ of subsets of $\NN$ and $0<p<\infty$ in $\NN$ we define the direct sum
\begin{equation}\label{eq:LpTsirelsonSpace}
\Ba^{(p)}_{\AI}=\left(\bigoplus_{j=1}^\infty L_p^{A_j}\right)_{\Ts^{(p)}}.
\end{equation}

\begin{theorem}\label{thm:IsoTsHaar}
Let $1<p<\infty$ and let $\AI=(A_j)_{j=1}^\infty$ be a sequence of nonempty subsets of $\NN$. Suppose that $\varphi(j):=\sup(A_j)<\min(A_{j+1})$ for each $j\in\NN$, and that the function $\varphi$ is dominated by the fast growing hierarchy. Then $\Ba^{(p)}_{\AI} \simeq \Ts^{(p)}$.
\end{theorem}

\begin{proof}
By \eqref{eq:DyadicEmbedding}, $\psi(j):=\dim (L_p^{A_j})\le 2^{\varphi(j)} \le \psi(j+1) $ for all $j\in\NN$. Therefore, by Part~\ref{FGH4} of Lemma~\ref{lem:FGHProp}, the function $\psi$ is increasing and dominated by the fast growing hierarchy. By \cite{Muller1988}*{Theorem 1}, there is a constant $C$ such that $L_p^{A_j}\simeq_C \ell_p^{\psi(j)}$ for all $j\in\NN$. Therefore, $\Ba^{(p)}_{\AI}\simeq ( \bigoplus_{j=1}^\infty \ell_p^{\psi(j)})_{\Ts^{(p)}}$. An application of Theorem~\ref{thm:TsirelsonIsomorphism} puts an end to the proof.
\end{proof}

For further applications of Theorem~\ref{thm:TsirelsonIsomorphism}, we will need the Schr\"oder-Bernstein principle for unconditional bases, which we enunciate in the language of lattices.

\begin{theorem}[see \cite{Wojtowicz1988}*{Corollary 1}]\label{thm:SBUB}
Let $\LL_1$ and $\LL_2$ be sequence spaces. Suppose that $\LL_1$ is isomorphic to a restriction of $\LL_2$ and, the other way around, $\LL_2$ is isomorphic to a restriction of $\LL_1$. Then, $\LL_1$ and $\LL_2$ are lattice isomorphic.
\end{theorem}

\begin{proposition}\label{prop:DSSquare}
Let $\LL$ be a sequence space. Suppose that $\LL$ is lattice isomorphic to its square via a bijection $\pi\colon \enbrace{1,2}\times \NN\to\NN$ such that $\pi(i,\cdot)$ is increasing, $i=1$, $2$. For each $j\in\NN$, let $\XX_j$ be a sequence space, and assume that there is a constant $C$ such that $\XX_j$ is lattice $C$-isomorphic to a restriction of $\XX_k$ for all $(j,k)\in\NN^2$ with $j\le k$. Then $\XX=(\bigoplus_{j=1}^\infty \XX_j)_\LL$ is lattice isomorphic to its square.
\end{proposition}

\begin{proof}
The sequence space $\XX^2$ is lattice isomorphic to the sequence space $(\bigoplus_{k=1}^\infty \XX_{\psi(k)})_\LL$, where $\psi(k)=j$ if $k=\pi(i,j)$ either for $i=1$ or $i=2$. Since $k=\pi(i,j)$ implies $j\le k$, we have $\psi(k)\le k$ for all $k\in\NN$. Consequently, $\XX^2$ is lattice isomorphic to a restriction of $\XX$. On the other hand, $\XX$ is obviously lattice isomorphic to a restriction of $\XX^2$. By Theorem~\ref{thm:SBUB}, $\XX^2$ is lattice isomorphic to $\XX$.
\end{proof}

Unlike the category of sequence spaces, the category of quasi-Banach spaces does not satisfy the Schr\"oder-Bernstein principle \cite{GowersMaurey1997}. Notwithstanding, Pelczy\'{n}ski's decomposition technique \cite{Pel1960} provides sufficient conditions for a pair of quasi-Banach spaces to satisfy the Schr\"oder-Bernstein property.

\begin{theorem}[see \cite{AlbiacKalton2016}*{Theorem 2.2.3}]\label{thm:PDT}
Let $\XX$ and $\YY$ be quasi-Banach spaces. Suppose that $\XX\unlhd_C\YY$, $\YY\unlhd_C\XX$, $\XX^2\simeq_X\XX$, and $\YY^2\simeq_C\YY$ for some constant $C$. Then $\XX\simeq_D\YY$, where $D$ only depends on $C$.
\end{theorem}

\begin{theorem}\label{thm:GeneralIso}
Let $0<p<\infty$. For each $j\in\NN$, let $\XX_j$ be a sequence space. Suppose that that there is a constant $C$ a function $\varphi$ dominated by the fast growing hierarchy such that
\begin{itemize}[leftmargin=*]
\item $\XX_j\unlhd_C \ell_p^{\varphi(j)}$ for all $j\in\NN$, and
\item $\XX_j$ is lattice $C$-isomorphic to a restriction of $\XX_k$ for all $(j,k)\in\NN^2$ with $j\le k$.
\end{itemize}
Then $\XX:=(\bigoplus_{j=1}^\infty \XX_j)_{\Ts^{(p)}}$ is isomorphic to $\Ts^{(p)}$.
\end{theorem}

\begin{proof}
It is clear that $\Ts^{(p)}\unlhd \XX$ and $\XX \unlhd \YY:= ( \bigoplus_{j=1}^\infty \ell_p^{\varphi(j)})_{\Ts^{(p)}}$. By Theorem~\ref{thm:TsirelsonIsomorphism}, $\YY\simeq \Ts^{(p)}$. Combining Theorem~\ref{thm:CSSQ} with Proposition~\ref{prop:DSSquare}, gives that $\Ts^{(p)} \oplus \Ts^{(p)}\simeq\Ts^{(p)} $ and that $\XX^2\simeq\XX$. Applying Theorem~\ref{thm:PDT} puts an end to the proof.
\end{proof}

We close the section by exhibiting an isomorphism, which we will apply in Section~\ref{Sect:AlmostGreedy}, between the $p$-convexified Tsirelson's space $\Ts^{(p)}$ for $1<p<\infty$, and certain direct sums involving finite-dimensional Hilbert spaces.

It is known that $\ell_2$ is a complemented subspace of $L_p$, $1<p<\infty$. In fact, the proof of this result of
Pelczy\'{n}ski relies on the behavior of the Rademacher functions and
gives the following.

\begin{theorem}[see \cite{Pel1960}*{Proposition 5}]\label{thm:PelRad}
Let $0<p<\infty$. There is a constant $C$ such that for each $A\subseteq\NN$ finite there is a block basic sequence of $\HB^{(p)}_A$ which is $C$-equivalent to the unit vector system of $\ell_2^{\abs{A}}$. Moreover, if $p>1$ this block basic sequence spans a $C$-complemented subspace. In particular, if $1<p<\infty$, there is a constant $C$ such that $\ell_2^j\unlhd_C \ell_p^{2^j}$ for all $j\in\NN$.
\end{theorem}

\begin{theorem}\label{thm:p>1Iso}
Let $1<p<\infty$. Let $\varphi$ and $\psi\colon\NN\to\NN$ be non-decrasing sequences dominated by the fast growing hierarchy. Then
\[
\left(\bigoplus\ell_2^{\varphi(j)} \oplus \ell_p^{\psi(j)}\right)_{\Ts^{(p)}}\approx \Ts^{(p)}.
\]
\end{theorem}

\begin{proof}
By Theorem~\ref{thm:PelRad},
\[
\ell_2^{\varphi(j)} \oplus \ell_p^{\psi(j)}\unlhd_C \ell_p^{\eta(j)}, \quad j\in\NN,
\]
where $\eta(j)=\psi(j)+ 2^{\varphi(j)}$ and $C$ is a constant independent of $j$. Since, by Parts~\ref{FGH2} and \ref{FGH4} of Lemma~\ref{lem:FGHProp}, $\eta$ is dominated by the fast growing hierarchy, applying Theorem~\ref{thm:GeneralIso} puts an end to the proof.
\end{proof}

%------------------------------------------------------------------------
\section{Non-equivalent greedy bases in the $p$-convexified Tsirelson's space $\Ts^{(p)}$ for $1<p<2$ and $2<p<\infty$}\label{sect:greedy}\noindent
%------------------------------------------------------------------------
Let us first record the well-known fact that the canonical basis of Tsirelson's space is democratic.

\begin{theorem}[see \cite{CasShu1989}*{Proposition I.2}]\label{thm:TsirelsonDem}
Every semi-normalized block basic sequence of $\Ts$, in particular the unit vector system $(\ts_{n})_{n=1}^{\infty}$ of $\Ts$, is democratic with fundamental function equivalent to $(m)_{m=1}^\infty$.
\end{theorem}

A straightforward application of Theorem~\ref{thm:TsirelsonDem} gives that the canonical basis of $\Ts^{(p)}$, $0<p<\infty$, is (unconditional and) democratic, hence greedy, with fundamental function equivalent to $(m^{1/p})_{m=1}^\infty$. In the case when $0<p\le 1$ or $p=2$, the unit vector system is the unique unconditional basis of $\Ts^{(p)}$ up to equivalence and permutation \cites{BCLT1985,Woj1997,CasKal1998}. So, these spaces have a unique greedy basis up to equivalence and permutation. In this section, we will prove that if $p\in(1,2)\cup(2,\infty)$ the situation is totally opposed. Namely, we will prove that $\Ts^{(p)}$ has uncountably many greedy bases which are unequivalent under permutation. To that end, following ideas from \cite{DHK2006}, we consider subbases of the Haar system consisting of full levels.

We start with a proposition that will become instrumental to determine the democracy and the fundamental functions of bases of $\Ts^{(p)}$ for $0<p<\infty$.

\begin{proposition}\label{prop:DemTsirelson}
Let $0<p<\infty$. Suppose $\varphi\colon\NN\to\NN$ is a function dominated by the fast growing hierarchy. For each $j\in\NN$ let $\XX_j$ be a $\varphi(j)$-dimensional Banach space with a basis $\XB_j=(\xx_{j,n})_{n\in\Nt_j}$. Suppose there is a constant $C\ge 1$ such that
\[
\frac{1}{C} \abs{A}^{1/p} \le \norm{ \Ind_{A}[\XB_j,\XX_j]} \le C \abs{A}^{1/p}, \quad A\subseteq \Nt_j,\, j\in\NN.
\]
Then $\XB:=(\bigoplus_{j=1}^\infty \XB_j)_{\Ts^{(p)}}$ is a democratic basis of $\XX:=(\bigoplus_{j=1}^\infty \XX_j)_{\Ts^{(p)}}$ with fundamental function equivalent to $(m^{1/p})_{m=1}^\infty$.
\end{proposition}

\begin{proof}
Let $A\subseteq\Nt:=\cup_{j=1}^\infty \enbrace{j}\times\Nt_j$ be finite. If $A=\cup_{j=1}^\infty \enbrace{j}\times A_j$ then,
\[
\frac{1}{C} \norm{ \Ind_A[\XB,\XX] } \le \nu(A):=\norm{ \left( \abs{A_j}^{1/p}\right)_{j=1}^\infty }_{\Ts^{(p)}} \le C \norm{ \Ind_A[\XB,\XX] }.
\]
Set $\YY= ( \bigoplus_{j=1}^\infty \ell_p^{\abs{A_j}})_{\Ts^{(p)}}$, so that
$\nu(A)=\norm{ \Ind_A[\EB_\Nt,\YY]}$. An application of Theorem~\ref{thm:TsirelsonIsomorphism} yields $B\subseteq \NN$ such that $\abs{B}=\abs{A}$
and
\[
\nu(A)\approx \norm{ \Ind_B[\EB_\NN,\Ts^{(p)}]}.
\]
Applying Theorem~\ref{thm:TsirelsonDem} puts an end to the proof.
\end{proof}

In view of Proposition~\ref{prop:DemTsirelson} it may appear that we just managed to find bases that have a fundamental function of a certain, wished-for kind. However, our next result shows that that does not happen by chance and that, actually, the fundamental functions of all superdemocratic bases in $\Ts^{(p)}$ are of the same order.

\begin{proposition}
Let $0<p<\infty$. Then every super-democratic sequence in $\Ts^{(p)}$ has fundamental function equivalent to $(m^{1/p})_{m=1}^\infty$.
\end{proposition}

\begin{proof}
Let $\ee_j^*\colon\FF^\NN \to \FF$ denote the $j$th coordinate functional. If $\XB=(\xx_n)_{n=1}^\infty$ is super-democratic, then it is norm-bounded. Therefore, applying Cantor's diagonal argument gives $\nu\colon\NN\to\NN$ increasing such that $(\xx_{\nu(k)})_{k=1}^\infty$ converges pointwise. Thus, the sequence $(\yy_k)_{k=1}^\infty$ given by $\yy_k=\xx_{\nu(2k)}-\xx_{\nu(2k-1)}$ converges pointwise to zero. Since, by super-democracy, $\inf_k\norm{ \yy_k}_{\Ts^{(p)}}>0$, appyling the gliding-hump technique, gives $\eta\colon\NN\to\NN$ increasing such that $\YB=(\yy_j)_{j=1}^\infty$ is equivalent to a semi-normalized block basic sequence. By Theorem~\ref{thm:TsirelsonDem}, $\YB$ is democratic with fundamental function equivalent to $(m^{1/p})_{m=1}^\infty$. It follows that $\udf[\XB,\XX](2m)\approx m^{1/p}$ for $m\in\NN$. Since $\udf[\XB,\XX](2m)\lesssim \udf[\XB,\XX](m)$ for $m\in\NN$, we are done.
\end{proof}

Recall that, given a measure space $(X,\Sigma,\mu)$, a subset $X\subseteq L_1(\mu)$ is said to be equi-integrable if for every $\epsilon>0$ there is $\delta=\delta(\varepsilon)>0$ such that $\abs{\int_A f \, d\mu}\le \epsilon$ for all $f\in X$ and all $A\in\Sigma$ with $\mu(A)\le\delta$. If this holds with a given function $\delta(\cdot)\colon(0,\infty)\to(0,\infty)$, we say that $X$ is \emph{equi-integrable with function $\delta(\cdot)$}.

\begin{lemma}\label{lem:key}
Let $(X,\Sigma,\mu)$ be a measure space, and let $F$ be a subspace of $\subseteq L_1(\mu)$ whose unit ball is equi-integrable set with function $\delta(\cdot)$. Then
\[
({1-\sqrt{\varepsilon}})( \norm{ f }_1+\norm{ g}_1) \le \norm{ f+g}_1 \le (1+\varepsilon+\sqrt{\varepsilon})( \norm{ f }_1+\norm{ g}_1)
\]
for every $f\in F$, every $\varepsilon\in(0,1)$, and every $g\in L_1(\mu)$ which is zero outside a set of measure at most $\delta(\varepsilon)$.
\end{lemma}

\begin{proof}
We just need to go over the lines of the proof of \cite{DHK2006}*{Lemma 2.1}.
\end{proof}

\begin{lemma}\label{lem:HaarEqui}
Given $n\in\NN\cup\enbrace{0}$, the unit ball of $L_1(\Sigma_n)$ is equi-integrable with function $\varepsilon\mapsto 2^{-n}\varepsilon$.
\end{lemma}

\begin{proof}
Let $f\colon[0,1]\to \FF$ be $\Sigma_n$-measurable, and let $A\subseteq[0,1]$ be measurable. We have
\[
\int_A \abs{f(t)} \, dt\le \abs{A} \, \norm{ f }_\infty\le \abs{A} 2^{n} \norm{ f }_1.\qedhere
\]
\end{proof}

\begin{lemma}\label{lem:DHK}
Let $0<p<\infty$. Let $\varphi\colon\NN\to\NN$ be a function with
\[
\varphi(i+1)\ge 2\varphi(i)+i+u+1, \quad i\in\NN,
\]
for some $u\in\NN$. Let $S\subseteq\cup_{i=1}^\infty \Dy_{\varphi(i)}$ be such that
\[
\abs{S\cap \Dy_{\varphi(i+1)}}\le 2 \abs{ \Dy_{\varphi(i)}}, \quad i\in\NN.
\]
Then, there are $0<C_1<1<C_2<\infty$ such that
\[
C_1\left( \sum_{I\in S} \abs{c_I}^p\right)^{1/p} \le \norm{ \sum_{I\in S} c_I h_I^{(p)} }\le C_2 \left( \sum_{I\in S} \abs{c_I}^p\right)^{1/p},
\]
for every eventually null family $(c_I)_{I\in S}$. Specifically, we can choose
\begin{align*}
C_1=C_1(u)&=\prod_{k=1}^\infty \left( 1-2^{-(k+u)/2}\right)^{1/p},\\
C_2=C_2(u)&=\prod_{k=1}^\infty \left( 1+2^{-(k+u)/2}+2^{-(k+j)}\right)^{1/p},
\end{align*}
so that $\lim_u C_1(u)=\lim_u C_2(u)=1$.
\end{lemma}

\begin{proof}
Argue as in the proof of \cite{DHK2006}*{Lemma 2.2}, taking into consideration Lemmas~\ref{lem:DHK} and \ref{lem:HaarEqui}.
\end{proof}

Given a sequence $\AI=(A_j)_{j=1}^\infty$ of subsets of $\NN$ and $1<p<\infty$, we denote by $\BB^{(p)}_{\AI}$ the natural unconditional basis of the space $\Ba^{(p)}_{\AI}$ defined as in \eqref{eq:LpTsirelsonSpace}, that is,
\[
\BB^{(p)}_{\AI}=\left( \bigoplus_{j=1}^\infty \HB^{(p)}_{A_j} \right)_{\Ts^{(p)}}.
\]

\begin{lemma}\label{lem:Tplp}
Let $1<p<\infty$, and let $\AI=(A_j)_{j=1}^\infty$ and $\BI=(B_j)_{j=1}^\infty$ be sequences of subsets of $\NN$ such that:
\begin{itemize}[leftmargin=*]
\item $\AI$ and $\BI$ consist of pairwise disjoint nonempty subsets,
\item $(\cup_{j=1}^\infty A_j) \cap (\cup_{j=1}^\infty B_j)$ is finite, and
\item there is an increasing function $\varphi\colon\NN\to\NN$ such that
\[
(\cup_{j=1}^\infty A_j) \cup (\cup_{j=1}^\infty B_j)\subseteq\varphi(\NN)
\]
and $\varphi(i+1)\ge \varphi(i)+i+2$ for all $i\in\NN$.
\end{itemize}
Suppose that a sequence $\XB=(\xx_k)_{k=1}^\infty$ is equivalent to a permutation of a subbasis of both $\BB^{(p)}_{\AI}$ and $\BB^{(p)}_{\BI}$. Then $\XB$ is equivalent to a permutation of a subbasis of the canonical basis of $\Ts^{(p)}(\ell_p)$.
\end{lemma}

\begin{proof}
For the sake of convenience, let us put $h_I=h_I^{(p)}$.
Let $(j_a, I_a)$ and $(j_b,I_b)\colon \NN\to \NN\times\Dy$ be one-to-one maps such that $\XB$ is equivalent to both
$(L_{j_a(k)} (h_{I_a(k)}))_{k=1}^\infty$ and $(L_{j_b(k)} (h_{I_b(k)}))_{k=1}^\infty$. We split $\NN$ is three sets as follows:
\begin{align*}
N_a&=\enbrace{k\in\NN \colon \abs{I_a(k)}} < \abs{I_b(k)},\\
N_b&=\enbrace{k\in\NN \colon \abs{I_b(k)}} < \abs{I_a(k)}\mbox{, and}\\
N_0&=\enbrace{k\in\NN \colon \abs{I_b(k)}} = \abs{I_a(k)}.
\end{align*}

For each $j\in\NN$, let $N_{a,j}=\enbrace{k\in N_a \colon j_a(k)=j}$. Then, the family $(I_a(k))_{k\in N_{a,j}}$ consists of distinct dyadic intervals, and
\[
\norm{ \sum_{k\in N_a} c_k \, \xx_k }
\approx \norm{ \left(\norm{ \sum_{k\in N_{a,j}} c_k h_{I_a(k)} }_p\right)_{j=1}^\infty }_{\Ts^{(p)}},
\quad (c_k)_{k=1}^\infty\in c_{00}(N_a).
\]
If $i\in \NN$ and $k\in N_{a,j}$ satisfy $\abs{I_a(k)}=2^{-\varphi(i+1) }$, then $\abs{I_b(k)}\ge 2^{-\varphi(i)}$. Since the family $(I_b(k))_{k\in N_{a,j}}$ also consist of distinct dyadic intervals,
\[
\abs{\enbrace{ k\in N_{a,j} \colon \abs{I_a(k)}=2^{-\varphi(i+1) }}}\le 2^{1+\varphi(i)} -1\le 2 \abs{\Dy_{\varphi(i)}}.
\]
Therefore, an application of Lemma~\ref{lem:key} gives
\[
\norm{ \sum_{k\in N_a} c_k \, \xx_k }
\approx \norm{ \left(\left( \sum_{k\in N_{a,j}} \abs{c_k}^p \right)^{1/p}\right)_{j=1}^\infty }_{\Ts^{(p)}},
\quad (c_k)_{k=1}^\infty\in c_{00}(N_a).
\]
We define the sets $N_{b,j}$ and $N_{0,j}$ analogously. Switching the roles of $\AI$ and $\BI$ we obtain
\[
\norm{ \sum_{k\in N_b} c_k \, \xx_k }
\approx \norm{ \left(\left( \sum_{k\in N_{b,j}} \abs{c_k}^p \right)^{1/p}\right)_{j=1}^\infty }_{\Ts^{(p)}},
\quad (c_k)_{k=1}^\infty\in c_{00}(N_b).
\]
In turn, since $N_0$ is finite,
\[
\norm{ \sum_{k\in N_0} c_k \, \xx_k }
\approx \norm{ \left(\left( \sum_{k\in N_{0,j}} \abs{c_k}^p \right)^{1/p}\right)_{j=1}^\infty }_{\Ts^{(p)}},
\quad (c_k)_{k=1}^\infty\in c_{00}(N_0).
\]
Using that $\XB$ and the unit vector system of $\Ts^{(p)}$ are both unconditional bases, we obtain that $\XB$ is equivalent to a permutation of the canonical basis of $(\bigoplus_{j=1}^\infty \ell_p^{\eta(j)})_{\Ts^{(p)}}$, where $\eta(j)=\abs{N_{a,j}}+\abs{N_{b,j}}+\abs{N_{0,j}}$.
\end{proof}

\begin{lemma}\label{lemBBSl2}
Let $1<p<\infty$, and let $\AI=(A_j)_{j=1}^\infty$ be a sequence of nonempty subsets of $\NN$ with $\sup_j \abs{A_j}=\infty$. Then, there is a constant $C$ such that $\BB^{(p)}_{\AI}$ has for every $n\in\NN$ a block basis sequence $C$-equivalent to the unit vector system of $\ell_2^n$.
\end{lemma}

\begin{proof}
It is a straightforward consequence of Theorem~\ref{thm:PelRad}.
\end{proof}

We are almost ready to prove the main result of this section. Before that, we bring up another property of Tsirelson's space that we will need.

\begin{proposition}\label{prop:convexconcav}
Let $0<p<\infty$. Then $\Ts^{(p)}$ is a $p$-convex lattice and a $r$-concave lattice for all $r>p$.
\end{proposition}

\begin{proof}
It suffices to prove that $\Ts$ satisfies a lower $r$-estimate for all $r>1$. But this is taken care of in \cite{CasShu1989}*{Proposition V.10}.
\end{proof}

\begin{proposition}\label{prop:NEGB}
Let $1<p<\infty$, $p\not=2$, and let $\AI=(A_j)_{j=1}^\infty$ and $\BI=(B_j)_{j=1}^\infty$ be sequences of subsets of $\NN$ such that:
\begin{itemize}[leftmargin=*]
\item $\AI$ and $\BI$ consist of pairwise disjoint nonempty subsets,
\item $\sup_j \abs{A_j}=\infty$,
\item $(\cup_{j=1}^\infty A_j) \cap (\cup_{j=1}^\infty B_j)$ is finite, and
\item the increasing sequence $\zeta\colon\NN\to\NN$   with range
\[
(\cup_{j=1}^\infty A_j) \cup (\cup_{j=1}^\infty B_j)
\]
satisfies $\zeta(i+1)\ge \zeta(i)+i+2$ for all $i\in\NN$.
\end{itemize}
Then $\BB^{(p)}_{\AI}$ and $\BB^{(p)}_{\BI}$ are not permutatively equivalent.
\end{proposition}

\begin{proof}
By Lemma~\ref{lemBBSl2}, if the lattice structure induced by $\BB^{(p)}_{\AI}$ is $q$-convex and $r$-concave, then $q\le 2 \le r$. Assume by contradiction that $\BB^{(p)}_{\AI}$ and $\BB^{(p)}_{\BI}$ are permutatively equivalent. Then, by Lemma~\ref{lem:Tplp}, $\BB^{(p)}_{\AI}$ is permutatively equivalent to a subbasis of $\Ts^{(p)}(\ell_p)$. Consequently, by Proposition~\ref{prop:convexconcav}, the lattice structure induced by $\BB^{(p)}_{\AI}$ is $p$-convex and $r$-concave for every $r>p$. We infer that $p\le 2 \le r$ for every $r>p$. This absurdity proves the result.
\end{proof}

\begin{theorem}
Given $1<p<\infty$, $p\not=2$, the space $\Ts^{(p)}$ has a continuum of permutatively non-equivalent greedy bases whose greedy constants are uniformly bounded and whose fundamental functions are uniformly equivalent to $(m^{1/p})_{m=1}^\infty$.
\end{theorem}

\begin{proof}
Let $\Ft$ be the uncountable set of functions provided by Theorem~\ref{thm:FGHContinuum}. Consider the  increasing functions $\alpha$, $\beta\colon\NN\to\NN$ given by
\[
\alpha(k)=5\cdot 2^{k-1}-2k-2,\quad \beta(k)=k^2, \quad k\in\NN.
\]
For each function $\varphi\in\Ft$ consider the sequence $\AI_\varphi=(A_{j,\varphi})_{j=1}^\infty$ of subsets of $\NN$ given by
\[
A_{j,\varphi}=\enbrace{ \alpha(\varphi(n)) \colon \beta(j-1)<n\le \beta(j)}.
\]
 
By construction, $\abs{A_{j,\varphi}}=2j-1$ for all $j\in\NN$, and
\[
A_\varphi:=\cup_{j=1}^\infty A_{j,\varphi}=\alpha(\varphi(\NN)).
\]
Hence, given different functions  $\varphi$, $\psi\in\Ft$,  the set $A_\varphi\cap A_\psi$  is finite. Moreover, since the function $\alpha$ satisfies $\alpha(k+1)\ge \alpha(k)+k+2$ for all $k\in\NN$, so does the increasing function whose range is $A_\varphi\cup A_\psi$.

Applying Proposition~\ref{prop:NEGB} gives that the set
\[
\BB^{(p)}_{\AI_\varphi}, \quad \varphi\in\Ft,
\]
is a family of pairwise permutatively non-equivalent bases. Since $\alpha$ and $\beta$ are dominated by the fast growing hierarchy by Part~\ref{FGH4} of Lemma~\ref{lem:FGHProp}, an application of Part\ref{FGH5} of Lemma~\ref{lem:FGHProp} and Theorem~\ref{thm:IsoTsHaar} gives that $\Ba^{(p)}_{\AI_\varphi}\simeq \Ts^{(p)}$ for all $\varphi\in\Ft$. Since the Haar system is a democratic basis of $L_p$ \cite{Temlyakov1998}, applying Proposition~\ref{prop:DemTsirelson} gives that $\BB^{(p)}_{\AI_\varphi}$ is democratic with fundamental function equivalent to $(m^{1/p})_{m=1}^\infty$.

Finally, we point out that all the isomorphisms and equivalences that we use in the proof that $\BB^{(p)}_{\AI_\varphi}$ is a greedy basis are uniform.
\end{proof}

%------------------------------------------------------------------------
\section{Conditional almost greedy bases in Tsirelson's space and its convexifications $\Ts^{(p)}$ for $0<p<\infty$}\label{Sect:AlmostGreedy}\noindent
%------------------------------------------------------------------------
In this section we give a neat application of Theorem~\ref{thm:TsirelsonIsomorphism} to the structure of the spaces $\Ts^{(p)}$ for $0<p<\infty$ in that they contain an uncountable set of mutually non-equivalent \emph{conditional} (i.e., not unconditional) almost greedy bases.

 Unconditional bases are a special kind of quasi-greedy bases and, although the converse is not true in general, quasi-greedy bases always retain vestiges of unconditionality. Every Banach space with a Schauder basis has a conditional basis by a classical result of Pelczy\'{n}ski and Singer \cite{PelSin1964}. Thus in order to get a more accurate information on the structure of a given space by means of its conditional bases, one needs to restrict the discussion on their existence by imposing certain distinctive properties. These additional properties are imported to Banach space theory from greedy approximation theory, where we find very important types of bases, such as quasi-greedy or almost greedy bases, which are suitable to implement the greedy algorithm and yet they need not be unconditional.

The topic of finding conditional quasi-greedy bases in a given Banach (or quasi-Banach) space is not a trivial matter and their existence cannot be taken for granted. For instance, it is known \cite{DKKT2003} that the only quasi-greedy basis of $c_{0}$ is the canonical unit vector system, which is unconditional.

In 2003, Dilworth et al.\ \cite{DKK2003} conceived a method that we call the \emph{Dilworth-Kalton-Kutzarova method}, or \emph{DKK-method} for short, for constructing conditional quasi-greedy bases in certain Banach spaces $\XX$. This technique heavily relies on the existence of a complemented subspace of $\XX$ with a symmetric or subsymmetric basic sequence.

Thus, to advance the state-of-the-art of the subject we must investigate the existence of conditional almost-greedy bases in spaces \emph{without} a complemented subsymmetric basic sequence, and the most important examples os this kind of spaces are Tsirelson's space and its convexifications.

In order to be able to show the non-equivalence of the almost greedy bases that we will construct, we will quantify the conditionality of a basis $\XB$ of a Banach (or quasi-Banach) space $\XX$ by means of their \emph{conditionality parameters},
\[
\unc_m=\unc_m[\XB,\XX]=\sup_{\abs{A}\le m} \norm{ S_A }, \quad m\in\NN,
\]
and
\[
\aunc_m=\aunc_m[\XB,\XX]=\sup \enbrace{ \norm{ S_A(f)} \colon \norm{f}=1, \, \supp(f)\subseteq\NN[m]},\, m\in\NN,
\]
where for finite $A\subseteq \NN$,
\[
S_A(f)=\sum_{n\in A} \xx_n^*(f)\, \xx_n,\quad f\in \XX.
\]
We have that $\aunc_m\le \unc_m$ for all $m$, and the basis $\XB$ is unconditional if and only if $\sup_m \aunc_m<\infty$ (or, equivalently, $\sup_m \unc_m<\infty$).

The conditionality parameters of a conditional quasi-greedy basis $\XB$ of a $p$-Banach space $\XX$, $0<p\le 1$, grow slowly, to the extent that they verify the estimate
\begin{equation*}
\unc_m[\XB]\lesssim (\log m)^{1/p}, \quad m\ge 2,
\end{equation*}
(see \cite{DKK2003}*{Lemma 8.2} and \cite{AAW2021b}*{Theorem 5.1}). Answering a question raised by Temlyakov et al.\ \cite{TemYangYe2011b}, several authors have studied whether this bound can be improved when we consider quasi-greedy bases in some special classes of of Banach spaces \cite{GW2014}. For instance, in \cite{AAGHR2015} it was shown that every quasi-greedy basis in a superreflexive Banach space verifies the better condition
\begin{equation*}
\unc_m[\XB]\lesssim (\log m)^{1-\epsilon}, \quad m\ge 2.
\end{equation*}
for some $0<\epsilon<1$, and this is optimal. In the recent articles \cites{AADK2019b,AAW2021b} the authors push forward the research in this direction considerably by providing a wide class of spaces with conditional almost greedy bases whose conditionality constants are ``as large as possible.'' In the same spirit, here we will exhibit the existence of almost greedy bases as conditional as possible in Tsirelson's space $\Ts$ and its convexifications $\Ts^{(p)}$. Since the techniques in the non-superreflexive case (i.e., for $0<p\le 1$) are different from the superreflexive case (i.e., for $1<p<\infty$) we have split our discussion in these two possible scenarios.

%------------------------------------------------------------------------
\subsection{Almost greedy bases in $\Ts^{(p)}$ for $0<p\le 1$.}
%------------------------------------------------------------------------
To tackle this case we take advantage of the recent techniques developed in \cite{AAW2021b}.

\begin{theorem}\label{thm:HCAGBple1}
Let $0<p\le 1$ and suppose that $\eta\colon[1,\infty)\to[1,\infty)$ is a concave increasing function. Then $\Ts^{(p)}$ has a conditional almost greedy basis $\XB$ with
\begin{itemize}
\item $\unc_m[\XB,\Ts^{(p)}]\approx \eta^{1/p}(\log m)$ for $m\ge 2$, and
\item $\udf[\XB,\Ts^{(p)}](m) \approx m^{1/p}$ for $m\in\NN$.
\end{itemize}
\end{theorem}

\begin{proof}
Pick an arbitrary increasing function $\varphi\colon\NN\to\NN$ dominated by the fast growing hierarchy. By \cite{AAW2021b}*{Theorem 4.1}, \cite{AAW2021b}*{Proposition 4.2}, \cite{AAW2021b}*{Proposition 4.4}, \cite{AAW2021b}*{Theorem 5.1}, \cite{AAW2021b}*{Proposition 5.2} and \cite{AAW2021b}*{Proposition 6.3}, there is a $p$-Banach space $\YY$ and an almost greedy basis $\YB=(\yy_n)_{n=1}^\infty$ of $\YY$ such that
\begin{itemize}[leftmargin=*]
\item $\unc_m[\YB,\YY]\approx \eta^{1/p}(m)$ for $m\in\NN$,
\item $\udf[\YB,\YY](m) \approx m^{1/p}$ for $m\in\NN$, and
\item $\YY_j:=[\yy_n \colon 1\le n \le j]$ is $C$-isomorphic to $\ell_p^{j}$ for some constant $C$  independent of $j$.
\end{itemize}
Set $\XB=\left( \bigoplus_{j=1}^\infty \YB^{(\varphi(j))}\right)_{\Ts^{(p)}}$ and $\XX=\left( \bigoplus_{j=1}^\infty \YY_{\varphi(j)}\right)_{\Ts^{(p)}}$.
Then, $\XB$ is a quasi-greedy basis of $\XX$ with $\unc_m[\XB,\XX]\approx \eta^{1/p}(m)$ for $m\in\NN$.
By Theorem~\ref{thm:TsirelsonIsomorphism},
\[
\XX\simeq \left( \bigoplus_{j=1}^\infty \ell_p^{\varphi(j)}\right)_{\Ts^{(p)}} \simeq \Ts^{(p)}.
\]
By Proposition~\ref{prop:DemTsirelson}, $\XB$ is democratic with fundamental function equivalent to $(m^{1/p})_{m=1}^\infty$.
\end{proof}

%------------------------------------------------------------------------
\subsection{Almost greedy bases in $\Ts^{(p)}$ for $1<p<\infty$}
%------------------------------------------------------------------------
In contrast to the case $0<p\le 1$, here we will appeal to the DKK method. In order make the method fit our purposes, we need to construct suitable conditional Schauder bases of Hilbert spaces, and to that end we will combine Harmonic analysis methods with a nowadays standard `rotating' technique.

Given sequences $\XB=(\xx_n)_{n=1}^\infty$ and $\YB=(\yy_n)_{n=1}^\infty$ in quasi-Banach spaces $\XX$ and $\YY$, the `rotated' sequence $\XB\diamond\YB=(\zz_n)_{n=1}^\infty$ in $\XX\oplus\YY$ is given by
\[
\zz_{2n-1}=\frac{1}{\sqrt{2}}(\xx_n,\yy_n), \quad \zz_{2n}=\frac{1}{\sqrt{2}}(\xx_n,-\yy_n), \quad n\in\NN
\]
(see, e.g., \cite{AABBL2021}*{\S4}).

\begin{theorem}\label{thm:GWBasis}
For each $0<a<1$ there is a Schauder basis $\HB_a$ of $\ell_2$ such that
\begin{itemize}
\item $\aunc_m[\HB_a,\ell_2]\gtrsim m^a$ for $m\in\NN$; and
\item given $b>a$, $\unc_m[\HB_a,\ell_2]\lesssim m^b$ for $m\in\NN$.
\end{itemize}
\end{theorem}

\begin{proof}
Consider the sequence $\XB=(\xx_n)_{n=1}^\infty$ of trigonometric polynomials (on the variable $t$) given by
\[
\xx_1=\frac{1}{\sqrt{2\pi}},\; \xx_{2j}(t)= \frac{1}{\sqrt{\pi}}\cos(j t), \; \xx_{2j+1}(t)= \frac{1}{\sqrt{\pi}}\sin(j t), \quad j\in\NN.
\]
It is known \cite{Babenko1948} that $\XB$ is a Schauder basis of $L_2([-\pi,\pi], \abs{t}^{\lambda} \, dt)$ for every $\lambda\in(-1,1)$. Moreover, if $A_m=\NN[2m+1]$,
\begin{equation}\label{eq:Dirichlet}
\norm{ \Ind_{A_m}[\XB, L_2([-\pi,\pi], \abs{t}^{\lambda} \, dt)] } \approx m^{(1-\lambda)/2}, \quad m\in\NN
\end{equation}
(see \cite{GW2014}*{Lemma 3.7}). Suppose that $\lambda>0$. By H\"older's inequality,
\[
L_2([-\pi,\pi], \abs{t}^{\lambda} \, dt)\subseteq L_p([-\pi,\pi]), \quad p<\frac{2}{1+\lambda}.
\]
Combining this embedding with Hausdorff-Young's inequality gives that $\XB$, regarded as a basis of $L_2([-\pi,\pi], \abs{t}^{\lambda} \, dt)$, dominates the unit vector system of $\ell_q$ for every $q>2/(1-\lambda)$. Since $\XB$ is an orthonormal basis of $L_2([-\pi,\pi])$, for $\lambda<0$ the trivial embedding
\[
L_2([-\pi,\pi], \abs{t}^{\lambda} \, dt)\subseteq L_2([-\pi,\pi])
\]
gives, that $\XB$ dominates the unit vector system of $\ell_2$. We infer that, if $0<a<1$, the rotated system $\YB:=\XB\diamond \XB$ is a Schauder basis of
\[
\HH=L_2([-\pi,\pi], \abs{t}^{a} \, dt) \oplus L_2([-\pi,\pi], \abs{t}^{-a} \, dt)
\]
which dominates the unit vector system of $\ell_q$ for every $q>2/(1-a)$.

Via the natural identification of the dual space of $L_2([-\pi,\pi], \abs{t}^{\lambda} \, dt)$ with $L_2([-\pi,\pi], \abs{t}^{-\lambda} \, dt)$, $-1<\lambda<1$, the dual basis of $\XB$ is the sequence $\XB$ itself. Consequently, the dual basis of $\YB$ is the sequence $\YB$ itself. Hence, by duality, $\YB$, regarded as a basis of $\HH$, is dominated by the unit vector system of $\ell_p$ for every $p<2/(1+a)$. An application of \cite{AAB2021}*{Lemma 5.1} gives
\[
\unc_m[\YB,\HH]\lesssim m^{b}, \quad m\in\NN,
\]
where $b$ is arbitrarily close to
\[
\frac{1+a}{2}-\frac{1-a}{2}=a.
\]
Set $\varepsilon=((-1)^n)_{n=1}^\infty$ and $B_m=\NN[4m+2]$. By Inequality~\eqref{eq:Dirichlet},
\[
\frac{\norm{ \Ind_{\varepsilon,B_m}[\YB,\HH] }}
{\norm{ \Ind_{B_m}[\YB,\HH] }}\approx m^{a}.
\]
We infer that $\aunc_m[\YB,\HH]\gtrsim m^a$ for $m\in\NN$.
\end{proof}

Let $(\XX,\norm{ \cdot}_\XX)$ be a Banach space with a semi-normalized Schauder basis $\XB=(\xx_n)_{n=1}^\infty$ and let $(\Sym,\norm{ \cdot}_\Sym)$ be a subsymmetric sequence space, i.e., a sequence space for which the unit vector system is a $1$-sub\-sym\-met\-ric basis. Set

\[
\Lambda_m=\norm{ \sum_{j=1}^m \ee_j}_\Sym, \quad m\in\NN.
\]
Let $\sigma=(\sigma_n)_{n=1}^\infty$ be an \emph{ordered partition} of $\NN$, i.e., a partition into integer intervals with
\[
\max(\sigma_n)<\max(\sigma_{n+1}),\quad n\in\NN.
\]
The \emph{averaging projection} $P_\sigma\colon\FF^\NN\to\FF^\NN$ associated with the ordered partition $\sigma$ can be expressed as
\[
P_\sigma(f)=\sum_{n=1}^\infty \vv_n^*(f) \, \vv_n,
\]
where, if $\ee_j^*$ is $j$-th coordinate functional, defined on $\FF^\NN$ by $(a_k)_{k=1}^\infty\mapsto a_j$,
\[
\quad \vv_n= \frac{1}{\Lambda_{\abs{\sigma_n}}} \sum_{j\in\sigma_n} \ee_j,\quad
\vv_n^*= \frac{\Lambda_{\abs{\sigma_n}}}{\abs{\sigma_n}} \sum_{j\in\sigma_n} \ee_j^*,
\quad n\in\NN.
\]

Let $Q_\sigma=\Id_{\FF^\NN}-P_\sigma$ be the complementary projection. We define $\norm{ \cdot }_{\XB,\Sym,\sigma}$ on $c_{00}$ by
\[
\norm{ f }_{\XB,\Sym,\sigma}=\norm{ Q_\sigma(f)}_\Sym+\norm{ \sum_{n=1}^\infty \vv_n^*(f)\, \xx_n}_\XX, \quad f\in c_{00}.
\]
The completion of the normed space $(c_{00},\norm{ \cdot }_{\XB,\Sym,\sigma})$ will be denoted by $\YY[\XB,\Sym,\sigma]$. Theorem~\ref{thm:AADKGathered} below gathers the properties of the DKK-method that we will need. To properly enunciate it, we need to introduce some terminology. Set
\[
M_r=\sum_{n=1}^r \abs{\sigma_n}, \quad r\in\NN.
\]
Given $m\in\NN$, let $\Sym^{(m)}$ and $\YY^{(m)}[\XB,\Sym,\sigma]$ respectively denote the $m$-dimensional subspaces of $\Sym$ and $\YY[\XB,\Sym,\sigma]$ spanned by $(\ee_n)_{n=1}^m$. Similarly, $\XX^{(m)}$ will denote the $m$-dimensional subspace of $\XX$ spanned by $(\xx_n)_{n=1}^m$, and $\Sym_\sigma^{(m)}$ will be the subspace of $\Sym$ spanned by $(\vv_n)_{n=1}^m$.

\begin{theorem}\label{thm:AADKGathered}
Let $\XX$ be a Banach space with a Schauder basis $\XB$, let $\Sym$ be a subsymmetric sequence space, and let $\sigma$ be an ordered partition of $\NN$.
\begin{enumerate}[label=(\roman*),leftmargin=*,widest=iii]
\item\label{it:AADK1} The unit vector system is a semi-normalized Schauder basis of $\YY[\XB,\Sym,\sigma]$.
\item\label{it:AADK2} For $r\in\NN$, the space $\YY^{(M_r)}[\XB,\Sym,\sigma]\oplus \Sym_\sigma^{(r)}$ is uniformly isomorphic to $\XX^{(r)} \oplus \Sym^{(M_r)}$.

\item\label{it:AADK3} Assume that $(\Lambda_n)_{n=1}^\infty$ has both the LRP and the URP, and that $M_r\lesssim M_{r+1}-M_r$ for $r\in\NN$. Then, the unit vector system is an almost greedy basis of $\YY[\XB,\Sym,\sigma]$ with fundamental function equivalent to $(\Lambda_n)_{n=1}^\infty$.

\item\label{it:AADK4} Assume that $\aunc_m[\XB,\XX]\gtrsim \eta(m)$ for $m\in\NN$ for some non-decreasing doubling function $\eta\colon[0,\infty)\to[0,\infty)$, and that $\log(M_r) \lesssim r$ for $r\in\NN$. Then
$
\aunc_{m}[\EB_\NN, \YY[\XB,\Sym,\sigma]]\gtrsim \eta(\log m)
$
for $m\in\NN$.

\item\label{it:AADK6}Suppose that the lattice $\Sym$ satisfies a lower $s$-estimate and an upper $t$-estimate, $1\le s \le t \le \infty$, that $\aunc_m[\XB]\lesssim \eta(m)$ for $m\in\NN$ for some non-decreasing doubling function $\eta\colon[0,\infty)\to[0,\infty)$, and that $ \log(M_r) \gtrsim r$ for $r\in\NN$.
Then
\[
\aunc_{m}[\EB,\YY[\XB,\Sym,\sigma]]\lesssim \max\enbrace{ (\log m)^{\max\enbrace{1/s,1-1/t}}, \eta(\log m)}, \quad m\ge 2.
\]

\item\label{it:AADK5} Assume that $(\Lambda_n)_{n=1}^\infty$ has both the LRP and the URP. Then,
\[
\aunc_{m}[\EB,\YY[\XB,\Sym,\sigma]]\approx \unc_{m}[\EB,\YY[\XB,\Sym,\sigma]], \quad m\in\NN.
\]
\end{enumerate}
\end{theorem}

\begin{proof}
Parts \ref{it:AADK1}, \ref{it:AADK2}, \ref{it:AADK3} and \ref{it:AADK4} were proved in \cite{AADK2019b} (see Theorem 3.6, Proposition 3.8, Theorem 3.17). To prove \ref{it:AADK6}, we argue as in the proof of \cite{AADK2019b}*{Lemma 3.14} replacing the triangle law with the lower $s$-estimate condition and obtain
\begin{equation}\label{eq:ImpEstimate}
\norm{ Q_\sigma(S_A(f))}
\le
5 \norm{ Q_\sigma(f)}_\Sym
+2 C_s \left(\sum_{n=1}^\infty \frac{\Lambda^s_{\abs{A\cap\sigma_n}}}{\Lambda^s_{\abs{\sigma_n}}} \abs{ \vv_n^*(f)}^s\right)^{1/s}\end{equation}
for all $A\subseteq\NN$ and all $f\in c_{00}$, where $C_s$ is the lower $s$-estimate constant of $\Sym$. Pick $r\in\NN$, $f\in\FF^\NN$ with $\supp(f)\subseteq\cup_{n=1}^r \sigma_n$, and $A\subseteq\NN$. Suppose first that $P_\sigma(f)=f$, so that $Q_\sigma(f)=0$. By \cite{AADK2019b}*{Lemma 3.11},
\[
\abs{\vv_n^*(S_A(f))}\le 2 \abs{\vv_n^*(f)}, \quad 1\le n \le r.
\]
Hence, by convexity,
\[
\norm{ \sum_{n=1}^\infty \vv_n^*(S_A(f))\, \xx_n}_\XX
\le 8 \, \aunc_r[\XB] \,
\norm{ \sum_{n=1}^\infty \vv_n^*(f)\, \xx_n}.
\]
Set $C_*=\sup_n \norm{ \xx_n^*}$. By inequality~\eqref{eq:ImpEstimate},
\[
\norm{ Q_\sigma(S_A(f)) }_\Sym \le 2 C_* C_s r^{1/s} \norm{ \sum_{n=1}^\infty \vv_n^*(f)\, \xx_n}_\XX.
\]
Suppose now that $P_\sigma(f)=0$, so that $Q_\sigma(f)=f$ and $\vv_n^*(f)=0$ for all $n\in\NN$. By inequality~\eqref{eq:ImpEstimate}, $\norm{ Q_\sigma(S_A(f))}_\Sym\le 5 \norm{ f}_\Sym$. In turn, if $C=\sup_n \norm{ \xx_n}$ and $C_t$ is the upper $t$-estimate constant of $\Sym$, using \cite{AADK2019b}*{Lemma 3.11} we obtain
\begin{align*}
\norm{ \sum_{n=1}^\infty \vv_n^*(S_A(f)) \, \xx_n}_\XX
& \le C\sum_{n=1}^\infty \abs{ \vv_n^*(S_A(f))}\\
& \le 2C\sum_{n=1}^\infty \norm{ S_{\sigma_n}(f) }_\Sym\\
&\le 2 C r^{1-1/t} \left(\sum_{n=1}^\infty\norm{ S_{\sigma_n}(f) }_\Sym^t\right)^{1/t}\\
&\le 2 C C_t r^{1-1/t} \norm{ f}_\Sym.
\end{align*}

To complete the proof of \ref{it:AADK6}, we pick $m\in\NN$, and $f\in\FF^\NN$ with $\supp(f)\subseteq\NN[m]$. Choose $r\in\NN$ with $m\in\sigma_n$, and set $f_1=P_\sigma(f)$ and $f_2=Q_\sigma(f)$. We have $\vv_n^*(f_1)=\vv_n^*(f)$ for all $n\in\NN$, and $P_\sigma(f_2)=0$, and $f=f_1+f_2$, and $\supp(f_1)\cup\supp(f_2)\subseteq\cup_{n=1}^r \sigma_n$. Therefore,
\begin{align*}
\norm{ S_A(f) }_{\XB,\Sym,\sigma}
& \le \norm{ S_A(f_1) }_{\XB,\Sym,\sigma}+ \norm{ S_A(f_2) }_{\XB,\Sym,\sigma}\\
& \le \max\enbrace{ 8 \, \aunc_r[\XB] + 2 C_* C_s r^{1/s} , 5+2 C C_t r^{1-1/t} } \norm{ f}_{\XB,\Sym,\sigma}.
\end{align*}
Since there is a constant $D$ such that $r\le D \log m$ for all $m\ge 2$, we are done.

Finally, we tackle the proof \ref{it:AADK5}. Pick $m\in\NN$, and we choose $r\in\NN$ with $m \in \sigma_r$, so that $M_{r-1}<m\le M_r$. Given $A\subseteq \NN$ with $\abs{A}\le m$, set $B= A\cap [m+1,\infty)$. We have $\abs{B} \le M_r$, $B\subseteq \cup_{n = r}^\infty \sigma_n$ and
\[
S_A(f)=S_A(S_m(f))+S_B(f), \quad f\in c_{00}.
\]
By \cite{AADK2019b}*{Lemma 3.16} and the already proved Part~ \ref{it:AADK1}, there is a constant $C$ such that
\[
\max\enbrace{ \norm{ S_m(f)}_{\XB,\Sym,\sigma}, \norm{ S_B(f)}_{\XB,\Sym,\sigma}}
\le C\norm{ f}_{\XB,\Sym,\sigma}, \quad f\in c_{00}.
\]
We deduce that $\unc_m\le C(1+\aunc_m)$.
\end{proof}

\begin{theorem}\label{thm:HCAGBp>1}
Let $1<p<\infty$. For any $a$ with $\max\enbrace{1/p, 1-1/p}\le a<1$, the space $\Ts^{(p)}$ has a conditional almost greedy Schauder basis $\XB_a$ such that:
\begin{itemize}
\item $\unc_m[\XB_a,\Ts^{(p)}]\approx\aunc_m[\XB_a,\Ts^{(p)}]\gtrsim (\log m)^a$ for $m\ge 2$;
\item given $b>a$, $\unc_m[\XB_a,\Ts^{(p)}]\lesssim (\log m)^b$ for $m\ge 2$; and
\item $\udf[\XB_a,\Ts^{(p)}](m) \approx m^{1/p}$ for $m\in\NN$.
\end{itemize}
\end{theorem}

\begin{proof}
Let $\sigma=(\sigma_n)_{n=1}^\infty$ be an ordered partition of $\NN$ with $M_r\lesssim M_{r+1}-M_r$ and
$\log(M_r) \lesssim r$ for $r\in\NN$. Since the unit vector system of $\ell_p$ is perfectly homogeneous, the canonical basis $(\vv_n)_{n=1}^\infty$ of $P_\sigma(\ell_p)$ is isometrically equivalent to the unit vector basis of $\ell_p$. Let $\HB$ be the semi-normalized Schauder basis whose existence is guaranteed by Theorem~\ref{thm:GWBasis}. By Theorem~\ref{thm:AADKGathered}, applying the DKK-method with $\sigma$, $\HB$, and the symmetric space $\ell_p$ yields an almost greedy Schauder basis $\YB=(\yy_n)_{n=1}^\infty$ (of some Banach space $\YY$) whose fundamental function and conditionality parameters are as desired. Moreover, there is a constant $C$ such that
\begin{equation}\label{eq:Glue}
\YY_j\oplus \ell_p^j\simeq_C \ell_2^j \oplus \ell_p^{\psi(j)}, \quad j\in\NN,
\end{equation}
where $\YY_j:=[\yy_n \colon 1 \le n\le \psi(j) ]$.
By Proposition~\ref{prop:DemTsirelson}, the natural arrangement of the basis vectors of the system
\[
\UB=\left(\bigoplus_{j=1}^\infty \YB^{(\psi(j))} \oplus \EB_{\NN[m]}\right)_{\Ts^{(p)}}
\]
is an almost greedy Schauder basis of
$
\UU=\left(\bigoplus_{j=1}^\infty \YY_j \oplus \ell_p^{j}\right)_{\Ts^{(p)}}
$
whose fundamental function and conditionality parameters are as desired. Since, by \eqref{eq:Glue} and Theorem~\ref{thm:p>1Iso},
\[
\UU
\simeq
\left(\bigoplus_{j=1}^\infty \ell_2^j \oplus \ell_p^{\psi(j)} \right)_{\Ts^{(p)}}
\simeq \Ts^{(p)},
\]
we are done.
\end{proof}

Given an almost greedy conditional basis $\XB=(\xx_n)_{n=1}^\infty$ and a permutation $\pi$ of $\NN$, the rearranged sequence $(\xx_{\pi(n)})_{n=1}^\infty$ is an almost greedy basis which needs not be equivalent to $\XB$. Similarly, if $(a_n)_{n=1}^\infty$ is a semi-normalized sequence in $\FF$, the `perturbed' sequence $(a_n\xx_n)_{n=1}^\infty$ is an almost greedy basis which is not necessarily equivalent to $\XB$. Thus, to study the structure of almost greedy bases of a given quasi-Banach space, we must consider equivalence up to permutation
and perturbation.

\begin{corollary}
The space $\Ts^{(p)}$, $0<p<\infty$, has uncountably many non-equivalent conditional almost greedy bases $(\XB_j)_{j\in J}$. To be precise, if $i\not=j$ no perturbation of $\XB_i$ is equivalent to a permutation of $\XB_j$.
\end{corollary}
\begin{proof}
Notice that if a perturbation of a basis $\XB$ is equivalent to a permutation of a basis $\YB$, then $\unc_m[\XB]\approx \unc_m[\YB]$ for $m\in\NN$. Thus, the result follows from Theorem~\ref{thm:HCAGBple1} in the case when $0<p<1$, and from Theorem~\ref{thm:HCAGBp>1} in the case when $1<p<\infty$.
\end{proof}

%\bibliography{BiblioGT}
%\end{document}


\end{document}